\documentclass[reqno]{amsart}
\usepackage{amsmath,amssymb,amsfonts,caption}
\usepackage{graphicx}
\usepackage{euscript}
\usepackage{enumerate}
\usepackage{verbatim}
\usepackage{epsfig}

\usepackage{color}
\usepackage[usenames,dvipsnames,svgnames,table]{xcolor}
\definecolor{orange}{rgb}{1,0.5,0}
\newtheorem{theorem}{Theorem}
\newtheorem{corollary}{Corollary}
\newtheorem{definition}{Definition}

\newtheorem{lemma}{Lemma}

\renewcommand{\epsilon}{\varepsilon}
\renewcommand{\phi}{\varphi}
\renewcommand{\le}{\leqslant}
\renewcommand{\ge}{\geqslant}
\newcommand{\eps}{\varepsilon}
\newcommand{\lbd}{\lambda}
\newcommand{\ds}{\, ds}
\newcommand{\dr}{\, dr}

\newcommand{\dint}{\displaystyle \int}

\newcommand{\cR}{\mathcal R}

\DeclareMathOperator{\e}{e}

\DeclareMathOperator{\imagem}{Im}
\DeclareMathOperator{\codim}{codim}

\usepackage{dsfont}
   \newcommand{\N}{\ensuremath{\mathds N}}
   \newcommand{\R}{\ensuremath{\mathds R}}

\begin{document}
\title[]
   {Periodic orbits for periodic eco-epidemiological systems with infected prey}
\author{Lopo F. de Jesus}
\address{L. F. de Jesus
   Departamento de Matem\'atica\\
   Universidade da Beira Interior\\
   6201-001 Covilh\~a\\
   Portugal}
   \email{lopo.jesus@ubi.pt}
\author{C\'esar M. Silva}
\address{C. M. Silva\\
   Departamento de Matem\'atica\\
   Universidade da Beira Interior\\
   6201-001 Covilh\~a\\
   Portugal}
\email{csilva@ubi.pt}
\author{Helder Vilarinho}
\address{H. Vilarinho\\
   Departamento de Matem\'atica\\
   Universidade da Beira Interior\\
   6201-001 Covilh\~a\\
   Portugal}
\email{helder@ubi.pt}
\date{\today}
\thanks{L. de Jesus, C. M. Silva and H. Vilarinho were partially supported by FCT through CMA-UBI (project UIDB/MAT/00212/2020).}
\begin{abstract}
We address the existence of periodic orbits for periodic eco-epide\-mi\-ological system with disease in the prey. To do it, we consider three main steps. Firstly we study a one parameter family of systems and obtain uniform bounds for the components of any periodic solution of these systems. Next, we make a suitable change of variables in our family of systems to establish the setting where we are able to apply Mawhin's continuation Theorem. Finally, we use Mawhin's continuation Theorem to obtain our result. Later on, we present two examples that include previous results in the literature and  some numerical simulations to illustrate our results.
\end{abstract}
\maketitle

\section{Introduction}
Eco-epidemiological models are ecological models that include infected compartments. In many situations, these models describe more accurately the real ecological system than models where the disease is not taken into account.

There is already a large number of works concerning eco-epidemiological models. To mention just a few recent works, we refer~\cite{Chakraborty-Das-Haldar-Kar-2015} where a mathematical study on disease persistence and extinction is carried out;~\cite{Bai-Xo-2018} where the authors study the global stability of a delayed eco-epidemiological model with holling type III functional response, and~\cite{Purnomo-Darti-Suryanto-2017} where an eco-epidemiological model with harvesting is considered.

One of the main concerns when studying eco-epidemiological models is to determine conditions under which one can predict if the disease persists or dies out. In mathematical epidemiology, these conditions are usually given in terms of the so called basic reproduction ratio $\mathcal{R}_0$, defined in~\cite{Diekmann-Heesterbeek-Metz-1990} for autonomous systems as the spectral radius of the next generation matrix.

In~\cite{Bacaer-Guernaoui-2006}, $\mathcal{R}_0$ was introduced for the periodic models, and later on, in~\cite{Wang-Zhao-JDDE-2008}, the definition of $\mathcal{R}_0$ was adapted to the study of periodic patchy models. In the recent article~\cite{Garrione-Rebelo-NARWA-2017} the theory in~\cite{Wang-Zhao-JDDE-2008} was used in the study of persistence of the predator in a general periodic predator-prey models.

When persistence is guaranteed, the obtention of conditions that assure the existence of periodic orbits for periodic eco-epidemiological models is an important issue in the deepening of the description of these models since these orbits correspond to situations where possibly there is some equilibrium in the described ecological system, reflected in the fact that the behaviour of the theoretical model is the same over the years. In~\cite{Silva-JMAA-2017} it was proved that there is an endemic periodic orbit for the periodic version of the model considered in~\cite{Niu-Zhang-Teng-AMM-2011} when the infected prey is permanent and some additional conditions are fulfilled, partially giving a positive answer to a conjecture in this last paper.

The models in~\cite{Niu-Zhang-Teng-AMM-2011} and~\cite{Silva-JMAA-2017} assume that there is no predation on uninfected preys. In spite of that, this assumption is not suitable for the description of many eco-epidemiological models. The main purpose of this paper is to present some results on the existence of an endemic periodic orbit for periodic eco-epidemiological systems with disease in the prey that generalize the systems in~\cite{Niu-Zhang-Teng-AMM-2011} and \cite{Silva-JMAA-2017} by including in the model a general function corresponding to the predation of uninfected preys. The proof of our result relies on Mawhin's continuation theorem. Following the approach in~\cite{Silva-JMAA-2017}, we begin by locating the components of possible periodic orbits for the one parameter family of systems that arise in Mawhin's result, allowing us to check that the conditions of that theorem are fulfilled. Although the main steps in our proof correspond to the ones in~\cite{Silva-JMAA-2017}, several additional nontrivial arguments are needed in our case. Additionally, there is also a substantial difference between our approach and the one in~\cite{Niu-Zhang-Teng-AMM-2011,Silva-JMAA-2017}.
In fact, we take as a departure point some prescribed behaviour of the uninfected subsystem, corresponding to the dynamics of preys and predators in the absence of disease: we will assume in this work that we have global asymptotic stability of solutions of some special perturbations of the bidimensional predator-prey system (the system obtained by letting $I=0$ in the first and third equations in~\eqref{eq:principal}). Thus, when applying our results to particular situations, one must verify that the underlying uninfected subsystem satisfies our assumptions. On the other hand, our approach allows us to construct an eco-epidemiological model from a previously studied predator-prey model (the uninfected subsystem) that satisfies our assumptions. This approach has the advantage of highlighting the link between the dynamics of the eco-epidemiological model and the dynamics of the predator-prey model used in its construction.

\section{A general eco-epidemiological model with disease in prey}

As a generalization of the model considered in~\cite{Silva-JMAA-2017}, a periodic version of the general non-autonomous model introduced in~\cite{Niu-Zhang-Teng-AMM-2011}, we consider the following periodic
eco-epidemiological model:
\begin{equation}\label{eq:principal}
\begin{cases}
S'=\Lambda(t)-\mu(t)S-a(t)f(S,P)P-\beta(t)SI\\
I'=\beta(t)SI-\eta(t)PI-c(t)I\\
P'=(r(t)-b(t)P)P+\gamma(t)a(t)f(S,P)P+\theta(t)\eta(t)PI
\end{cases},
\end{equation}
where $S$, $I$ and $P$ correspond, respectively, to the susceptible prey, infected prey and predator, $\Lambda(t)$ is the recruited rate of the prey population, $\mu(t)$ is the natural death rate of the prey population, $a(t)$ predation rate of susceptible prey, $\beta(t)$ is the incidence rate, $\eta(t)$ is the predation rate of infected prey, $c(t)$ is the death rate in the infective class ($c(t) \ge \mu(t)$), $\gamma(t)$ is the
rate converting susceptible prey into predator (biomass transfer), $\theta(t)$ is the rate of converting infected prey into predator, $r(t)$ and $b(t)$ are parameters related the vital dynamics of the predator population that is assumed to follow a logistic law and includes the intra-
specific competition between predators. It is assumed that only susceptible preys $S$ are capable of reproducing, i.e, the infected prey is removed by death (including natural and disease-related death) or by predation before having the possibility of reproducing.

Given a $\omega$-periodic function $f$ we will use throughout the paper the notations $f^\ell=\inf_{t \in (0,\omega]} f(t)$, $f^u=\sup_{t \in (0,\omega]} f(t)$ and $\bar{f}=\frac{1}{\omega} \int_0^\omega f(s) \ds$.
We will assume the following structural hypothesis concerning the parameter functions and the function $f$ appearing in our model:
\begin{enumerate}[S$1$)]
\item \label{cond-1} The real valued functions $\Lambda$, $\mu$, $\beta$, $\eta$, $c$, $\gamma$, $r$, $\theta$ and $b$ are periodic with period $\omega$, nonnegative and continuous;
\item \label{cond-2} Function $f$ is nonnegative and continuous;
\item \label{cond-3} Function $x \mapsto f(x,z)$ is nondecreasing;
\item \label{cond-4} Function $z \mapsto f(x,z)$ is nonincreasing;
\item \label{cond-5} For all $(x,z)$ we have $\bar\beta \frac{\partial f}{\partial x}(x,z)+\eta \frac{\partial f}{\partial z}(x,z)\ge 0$;
\item \label{cond-6} For any $C_1,C_2>0$, function $x\mapsto f(C_1 x+C_2,x)$ is nondecreasing;
\item \label{cond-7} $\bar\Lambda>0$, $\bar\mu>0$, $\bar r>0$ and $\bar b>0$.
\item \label{cond-8} There is $\alpha\ge 1$ and $K>0$ such that $f(x,0) \le Kx^\alpha$.
\end{enumerate}

To formulate our next assumptions we need to consider two auxiliary equations and one auxiliary system. First, for each $\lambda \in (0,1]$, we need to consider the following equations:
\begin{equation}\label{eq:auxiliary-S(t)}
x'=\lbd(\Lambda(t)-\mu(t)x)
\end{equation}
and
\begin{equation}\label{eq:auxiliary-P(t)}
z'=\lbd(r(t)-b(t)z)z.
\end{equation}
Note that, if we identify $x$ with the susceptible prey population, equation~\ref{eq:auxiliary-S(t)} gives the behavior of the susceptible preys in the absence of infected preys and predator and identifying $z$ with the predator population, equation~\ref{eq:auxiliary-P(t)} gives the behavior of the predator in the absence of preys.

Equations~\eqref{eq:auxiliary-S(t)} and~\eqref{eq:auxiliary-P(t)} have a well known behavior, given in the following lemmas:
\begin{lemma}[Lemma 1 in~\cite{Niu-Zhang-Teng-AMM-2011}]\label{lm1-1dim-1}
For each $\lbd \in (0,1]$ there is a unique $\omega$-periodic solution of equation~\eqref{eq:auxiliary-S(t)}, $x_\lbd^*(t)$, that is globally asymptotically stable in $\R^+$.
\end{lemma}
\begin{lemma}[Lemma 2 in~\cite{Niu-Zhang-Teng-AMM-2011}]\label{lm1-1dim-2}
For each $\lbd \in (0,1]$ there is a unique $\omega$-periodic solution of equation~\eqref{eq:auxiliary-P(t)}, $z_\lbd^*(t)$, that is globally asymptotically stable in $\R^+$.
\end{lemma}

For each $\lambda \in (0,1]$, we also need to consider the next family of systems, which correspond to behavior of the preys and predators in the absence of infected preys (system~\eqref{eq:principal} with $I = 0$, $S = x$ and $P = z$):
\begin{equation}\label{eq:auxiliary-system-SP}
\begin{cases}
x'=\lbd(\Lambda(t)-\mu(t)x-a(t)f(x,z)z-\eps_1 x)\\
z'=\lbd(\gamma(t)a(t)f(x,z)+r(t)- b(t)z+\eps_2)z
\end{cases}.
\end{equation}
We now make our last structural assumption on system~\eqref{eq:principal}:
\begin{enumerate}[S$1$)]
\setcounter{enumi}{8}
\item\label{cond-9} For each $\lambda \in (0,1]$ and each $\eps_1,\eps_2\ge 0$ sufficiently small, system~\eqref{eq:auxiliary-system-SP} has a unique $\omega$-periodic solution, $(x^*_{\lbd,\eps_1,\eps_2}(t),z^*_{\lbd,\eps_1,\eps_2}(t))$,  with $x^*_{\lbd,\eps_1,\eps_2}(t)>0$ and $z^*_{\lbd,\eps_1,\eps_2}(t)>0$, that is globally asymptotically stable in the set $\{(x,z)\in (\R_0^+)^2: x\ge 0 \ \wedge \ z>0\}$. We assume that $(\eps_1,\eps_2) \mapsto (x^*_{\lbd,\eps_1,\eps_2}(t), \, z^*_{\lbd,\eps_1,\eps_2}(t))$ is continuous.
\end{enumerate}

Denoting $x^*_{\lbd}=x^*_{\lbd,0,0}$ and $z^*_{\lbd}=z^*_{\lbd,0,0}$, we introduce the numbers
\begin{equation}\label{def-R_0}
\overline{\mathcal R}_0=\frac{\bar \beta \bar \Lambda/\bar \mu}{\bar c+\bar \eta \bar r/\bar b}, \quad \mathcal R_0^\lbd=\frac{\overline{\beta x^*_{\lbd}}}{\overline{c}+\overline{\eta z^*_{\lbd}}} \quad \text{and} \quad \widetilde{\mathcal R}_0 =\inf_{\lbd \in (0,1]} \mathcal{R}_0^\lbd
\end{equation}

\section{Main result}
We now present our main result.
\begin{theorem}\label{the:main}
If $\widetilde{\cR}_0>1$, $\overline{\gamma a}\,\overline{\beta}-\overline{\theta\eta}\overline{a}\le 0$ and
\begin{equation}\label{cond:teo}
\overline{R}_0>1+\dfrac{\overline{\beta}}{\overline{\mu}}\left(\dfrac{\overline{a}\,\overline{r}}{\overline{\eta}\,\overline{r}+\overline{c}\overline{b}}-\dfrac{\overline{\gamma a}}{\overline{\theta \eta}}\right) f\left(\frac{\overline{\eta}\,\overline{r}/\overline{b}+\overline{c}}{\overline{\beta}},\overline{r}/\overline{b}\right)
\end{equation}
then system~\eqref{eq:principal} possesses an endemic periodic orbit of period $\omega$.
\end{theorem}

Our proof relies on an application of Mawhin's continuation theorem. We will proceed in several steps. Firstly, in subsection~\ref{subsection:Proof-subsection1}, we consider a one parameter family of systems and obtain uniform bounds for the components of any periodic solution of these systems. Next, in subsection~\ref{subsection:Proof-subsection2} we make a suitable change of variables in our family of systems to establish the setting where we will apply Mawhin's continuation Theorem. Finally, in subsection~\ref{subsection:Proof-subsection3}, we use Mawhin's continuation Theorem to obtain our result.
\subsection{Uniform Persistence for the periodic orbits of a one parameter family of systems.}\label{subsection:Proof-subsection1}
In this section, to obtain uniform bounds for the components of any periodic solution of the family of systems that we can obtain multiplying the right hand side of~\eqref{eq:principal} by $\lambda \in (0,1]$, we need to consider the auxiliary systems:
\begin{equation}\label{eq:principal-aux}
\begin{cases}
S_\lbd'=\lbd(\Lambda(t)-\mu(t)S_\lbd-a(t)f(S_\lbd,P_\lbd)P_\lbd-\beta(t)S_\lbd I_\lbd)\\
I_\lbd'=\lbd(\beta(t)S_\lbd I_\lbd-\eta(t)P_\lbd I_\lbd-c(t)I_\lbd)\\
P_\lbd'=\lbd(\gamma(t)a(t)f(S_\lbd,P_\lbd)P_\lbd+\theta(t)\eta(t)P_\lbd I_\lbd+ r(t)P_\lbd-b(t)P_\lbd^2)
\end{cases}.
\end{equation}
We will consider separately each of the several components of any periodic orbit.
\begin{lemma}\label{lemma:subsection-persist-1}
Let $x_\lbd^*(t)$ be the unique solution of~\eqref{eq:auxiliary-S(t)}. There is $L_1>0$ such that, for any $\lambda \in (0,1]$ and any periodic solution
$(S_\lbd(t),I_\lbd(t),P_\lbd(t))$ of~\eqref{eq:principal-aux} with initial conditions $S_\lbd(t_0)=S_0>0$, $I_\lbd(t_0)=I_0>0$ and $P_\lbd(t_0)=P_0>0$, we have $S_\lbd(t)+I_\lbd(t)\le x^*_\lbd(t) \le \Lambda^u/\mu^\ell$ and $S_\lbd \ge L_1$, for all $t \in \R$.
\end{lemma}

\begin{proof}
Let $(S_\lbd(t),I_\lbd(t),P_\lbd(t))$ be some periodic solution of~\eqref{eq:principal-aux} with initial conditions $S_\lbd(t_0)=S_0>0$, $I_\lbd(t_0)=I_0>0$ and $P_\lbd(t_0)=P_0>0$. Since $c(t) \ge \mu(t)$, we
have, by the first and second equations of~\eqref{eq:principal-aux},
$$(S_\lbd + I_\lbd)' \le \lbd\Lambda(t)-\lbd\mu(t) S_\lbd -\lbd c(t)I_\lbd \le \lbd\Lambda(t)-\lbd\mu(t)(S_\lbd + I_\lbd).$$
Since, by Lemma~\ref{lm1-1dim-1}, equation~\eqref{eq:auxiliary-S(t)} has a unique periodic orbit, $x^*_\lbd(t)$, that is globally
asymptotically stable, we conclude that $S_\lbd(t) + I_\lbd(t)\le x^*_\lbd(t)$ for all $t \in \R$. Comparing equation~\eqref{eq:auxiliary-S(t)} with equation $x'=\lbd\Lambda^u-\lbd\mu^\ell x$, we conclude that $x^*_\lbd(t)\le \Lambda^u/\mu^\ell$.

Using conditions~S\ref{cond-3}) and S\ref{cond-4}), by the third equation of~\eqref{eq:principal-aux}, we have
$$P_\lbd'\le\lbd(r(t)+\gamma(t)a(t)f(x_\lbd^*(t),0)+\theta(t)\eta(t)x_\lbd^*(t)-b(t)P_\lbd)P_\lbd\le
(\Theta^u- b^\ell P_\lbd)P_\lbd,$$
where function $\Theta$ is given by $\Theta(t)=r(t)+\gamma(t)a(t)f(x_\lbd^*(t),0)+\theta(t)\eta(t)x_\lbd^*(t)$.
 Thus, comparing with equation~\eqref{eq:auxiliary-P(t)} and using Lemma~\ref{lm1-1dim-2}, we get $P_\lbd(t)\le P_\lbd^*(t) \le \Theta^u/b^\ell$. Using
the bound obtained above, since $-\beta(t)S_\lbd(t)\ge-\beta(t)x^*_\lbd(t)$, we have, by conditions~S\ref{cond-3}), S\ref{cond-4}) and S\ref{cond-8}),
\[
\begin{split}
S_\lbd'
& =\lbd\Lambda(t)-\lbd\mu(t)S_\lbd-\lbd a(t)f(S_\lbd,P_\lbd)P_\lbd-\lbd \beta(t)S_\lbd I_\lbd\\
& \ge \lbd\Lambda^\ell-\left( \lbd\mu^u+\lbd a^u \frac{f(S_\lbd,0)}{S_\lbd}\frac{\Theta^u}{b^\ell} +\lbd \beta^u (x^*_\lbd)^u \right)S_\lbd\\
& \ge \lbd\Lambda^\ell-\left( \lbd\mu^u +\lbd a^u K((x^*_\lbd)^u)^{\alpha-1}\Theta^u/b^\ell +\lbd \beta^u (x^*_\lbd)^u \right)S_\lbd
\end{split}
\]
and thus
\[
S_\lbd(t) \ge \dfrac{ \lbd\Lambda^\ell}{\lbd\mu^u +\lbd a^u K((x^*_\lbd)^u)^{\alpha-1} \Theta^u/b^\ell +\lbd \beta^u (x^*_\lbd)^u}=:L_1.
\]
\end{proof}

\begin{lemma}\label{lemma:subsection-persist-2}
Let $z_\lbd^*(t)$ be the unique solution of~\eqref{eq:auxiliary-P(t)}. There is $L_2>0$ such that, for any $\lambda \in (0,1]$ and any periodic solution
$(S_\lbd(t),I_\lbd(t),P_\lbd(t))$ of~\eqref{eq:principal-aux} with initial conditions $S_\lbd(t_0)=S_0>0$, $I_\lbd(t_0)=I_0>0$ and $P_\lbd(t_0)=P_0>0$, we have $r^\ell/b^u\le z^*_\lbd(t)\le P_\lbd(t) \le L_2$, for all $t \in \R$.
\end{lemma}

\begin{proof}
Let $\lbd \in (0,1]$ and $(S_\lbd(t),I_\lbd(t),P_\lbd(t))$ be any periodic solution of~\eqref{eq:principal-aux} with initial conditions $S_\lbd(t_0)=S_0>0$, $I_\lbd(t_0)=I_0>0$ and $P_\lbd(t_0)=P_0>0$. We have
$$P_\lbd'=\lbd P_\lbd(\gamma(t)a(t)f(S_\lbd,P_\lbd)+\theta(t)\eta(t) I_\lbd+ r(t)-b(t)P_\lbd)\ge (\lbd r(t)-\lbd b(t)P_\lbd)P_\lbd.$$
Comparing the previous inequality with equation~\eqref{eq:auxiliary-P(t)} and using Lemma~\ref{lm1-1dim-2}, we get $P_\lbd(t)\ge z^*_\lbd(t)$. Moreover, comparing equation~\eqref{eq:auxiliary-P(t)} with equation $z'=\lbd(r^\ell-b^uz)z$ we conclude that $z^*_\lbd(t) \ge r^\ell/b^u$.

Using the computations in proof of the previous lemma, we have $P_\lbd(t)\le L_1$ and we take $L_2=L_1$.
\end{proof}

\begin{lemma}\label{lemma:subsection-persist-3}
Let $\widetilde{\mathcal R}_0>1$. There are $L_3,L_4>0$ such that, for any $\lambda \in (0,1]$ and any periodic solution
$(S_\lbd(t),I_\lbd(t),P_\lbd(t))$ of~\eqref{eq:principal-aux} with initial conditions $S_\lbd(t_0)=S_0>0$, $I_\lbd(t_0)=I_0>0$ and $P_\lbd(t_0)=P_0>0$, we have $L_3 \le I_\lbd(t) \le L_4$, for all $t \in \R$.
\end{lemma}

\begin{proof}
We will first prove that there is $\eps_1>0$ such that, for any $\lbd\in (0,1]$, we have
\begin{equation}\label{eq:contradicao-persist-I}
\limsup_{t \to +\infty} I_\lbd(t)\ge \eps_1.
\end{equation}
By contradiction, assume that~\eqref{eq:contradicao-persist-I} does not hold. Then, for any $\eps>0$, there must
be $\lbd>0$ such that $I_\lbd(t)<\eps$ for all $t \in \R$. We have
\[
\begin{cases}
S_\lbd'\le \lbd\Lambda(t)-\lbd\mu(t)S_\lbd-\lbd a(t)f(S_\lbd,P_\lbd)P_\lbd\\
P_\lbd'\le \lbd(\gamma(t)a(t)f(S_\lbd,P_\lbd)+ r(t)-b(t)P_\lbd+\lbd\eps\theta^u\eta^u)P_\lbd
\end{cases}
\]
and
\[
\begin{cases}
S_\lbd'\ge \lbd\Lambda(t)-\lbd\mu(t)S_\lbd-\lbd a(t)f(S_\lbd,P_\lbd)P_\lbd-\eps\lbd\beta^uS_\lbd\\
P_\lbd'\ge \lbd(\gamma(t)a(t)f(S_\lbd,P_\lbd)+ r(t)-b(t)P_\lbd)P_\lbd
\end{cases}.
\]
By condition~S\ref{cond-9}), we conclude that
$$x^*_{\lbd,\eps\lbd\beta^u,0}(t)\le S_\lbd(t)\le
x^*_{\lbd,0,\eps\lbd\theta^u\eta^u}(t)$$
and
$$z^*_{\lbd,\eps\lbd\beta^u,0}(t)\le P_\lbd(t)\le
z^*_{\lbd,0,\eps\lbd\theta^u\eta^u}(t).$$
Thus, using condition~S\ref{cond-9}), we have
\begin{equation}
\begin{split}
I_\lbd'
& =\lbd(\beta(t)S_\lbd -\eta(t)P_\lbd-c(t))I_\lbd\\
& \ge(\lbd\beta(t)x^*_{\lbd,\eps\lbd\beta^u,0}(t) -\lbd\eta(t)z^*_{\lbd,0,\eps\lbd\theta^u\eta^u}(t)
-\lbd c(t))I_\lbd\\
& \ge(\lbd\beta(t)x^*_{\lbd}(t) -\lbd\eta(t)z^*_{\lbd}(t)-\lbd c(t)-\phi(\eps))I_\lbd,
\end{split}
\end{equation}
where $\phi$ is a nonegative function such that $\phi(\eps)\to 0$ as $\eps \to 0$ (notice that, by continuity, we can assume that $\phi$ is
independent of $\lbd$ and, by periodicity of the parameter functions, it is independent of $t$).

Integrating in $[0,\omega]$ and using~\eqref{cond-9}, we get
\[
\begin{split}
0 & =\frac{1}{\omega}\left(\ln I_\lbd(\omega)-\ln I_\lbd(0)\right)
=\frac{1}{\omega}\int_0^\omega I_\lbd'(s)/I_\lbd(s)\ds \\
& \ge
\lbd\left(\overline{\beta x^*_{\lbd}}-\bar c-\overline{\eta z^*_{\lbd}}\right)+\phi(\eps)
= \lbd(\bar c+\overline{\eta z^*_{\lbd}})(\cR_0^\lbd-1)+\phi(\eps)
\end{split}
\]
and since
$$\cR_0^\lbd\ge \inf_{\ell \in (0,1]} \cR_0^\ell=\widetilde{\cR}_0>1,$$
we have a contradiction. We conclude that~\eqref{eq:contradicao-persist-I} holds.
Next we will prove that there is $\eps_2>0$ such that, for any $\lbd\in (0,1]$, we have
\begin{equation}\label{eq:contradicao-persist-I-2}
\liminf_{t \to +\infty} I_\lbd(t)\ge \eps_2.
\end{equation}
Assuming by contradiction that~\eqref{eq:contradicao-persist-I-2} does not hold, we conclude that there is a sequence $(\lbd_n,I_{\lbd_n}(s_n),I_{\lbd_n}(t_n))\subset (0,1]\times\R_0^+\times\R_0^+$ such that $s_n<t_n$, $t_n-s_n\le \omega$,
$$I_{\lbd_n}(s_n)=1/n, \quad I_{\lbd_n}(t_n)=\eps_2/2 \quad \text{and} \quad I_{\lbd_n}(t)\in (1/n,\eps_2/2), \ \text{for all}  \ t \in (s_n,t_n).$$
Since $\lbd_n\le 1$, by Lemma~\ref{lemma:subsection-persist-1} we have
$$I'_{\lbd_n}=({\lbd_n}\beta(t)S_{\lbd_n} -{\lbd_n}\eta(t)P_{\lbd_n}-{\lbd_n}c(t))I_{\lbd_n}\le \beta^u \Lambda^u I_{\lbd_n}/\mu^\ell$$
and thus
$$\ln(\eps_2 n/2)=\ln(I_{\lbd_n}(t_n)/I_{\lbd_n}(s_n))
=\int_{s_n}^{t_n} I'_{\lbd_n}(s)/I_{\lbd_n}(s) \ds \le \beta^u \Lambda^u \omega/\mu^\ell,$$
which is a contradiction since the sequence $(\ln(\eps_2 n/2))_{n\in \N}$ goes to $+\infty$ as $n \to +\infty$, and thus is not bounded.

We conclude that there is $\eps_2>0$ such that~\eqref{eq:contradicao-persist-I-2} holds. Letting $L_3=\eps_2$, we obtain $I_\lbd(t)\ge L_3$ for all $\lbd \in (0,1]$.

Since $I_\lbd(t)\le S_\lbd(t)+I_\lbd(t)$, by Lemma~\ref{lemma:subsection-persist-1}, we can take $L_4=L_2$ and the result is established.
\end{proof}
\subsection{Setting where Mawhin's continuation theorem will be applied.}\label{subsection:Proof-subsection2}
To apply Mawhin's continuation theorem to our model we make the change of variables: $S(t)=e^{u_1(t)}$, $I(t)=e^{u_2(t)}$ and $P(t)=e^{u_3(t)}$. With this change of variables, system~\eqref{eq:principal} becomes
\begin{equation}\label{eq:principal-aux-2}
\begin{cases}
u_1'=\Lambda(t)e^{-u_1}-a(t)f(e^{u_1},e^{u_3})e^{u_3-u_1}-\beta(t)e^{u_2}-\mu(t)\\
u_2'=\beta(t)e^{u_1}-\eta(t)e^{u_3}-c(t)\\
u_3'=\gamma(t)a(t)f(e^{u_1},e^{u_3})+\theta(t)\eta(t)e^{u_2}-b(t)e^{u_3}+r(t)
\end{cases}.
\end{equation}

Note that, if $(u_1^*(t),u_2^*(t),u_3^*(t))$ is an $\omega$-periodic solution of the system~\eqref{eq:principal-aux-2} then $(e^{u_1(t)},e^{u_2(t)},e^{u_3(t)})$ is an $\omega$-periodic solution of system~\eqref{eq:principal}.

 To define the operators in Mawhin's theorem (see appendix~\ref{appendix:MCT}), we need to consider the Banach spaces $(X,\|\cdot\|)$ and $(Z,\|\cdot\|)$ where $X$ and $Z$ are the space of $\omega$-periodic continuous functions $u:\R \to \R^3$:
 $$X=Z=\{u=(u_1,u_2,u_3) \in C(\mathbb{R},\mathbb{R}^3): u(t)=u(t+\omega)\}$$
 and
 $$\|u\|=\max_{t \in [0,\omega]}|u_1(t)|+\max_{t \in [0,\omega]}|u_2(t)|+\max_{t \in [0,\omega]}|u_3(t)|.$$

Next, we consider the linear map $\mathcal{L}: X \cap C^1(\mathbb{R},\mathbb{R}^3) \rightarrow Z$ given by \begin{equation}\label{eq:linear-map}
\mathcal{L}u(t)=\frac{du(t)}{dt}
\end{equation}
and the map $\mathcal{N}: X \rightarrow Z$ defined by
\begin{equation}\label{eq:linear-map-L-comp}
\mathcal{N}u(t)=\left[\begin{array}{l}
\Lambda(t)e^{-u_1(t)}-a(t)f(e^{u_1(t)},e^{u_3(t)})e^{u_3(t)-u_1(t)}-\beta(t)e^{u_2(t)}-\mu(t)\\
\beta(t)e^{u_1(t)}-\eta(t)e^{u_3(t)}-c(t)\\
\gamma(t)a(t)f(e^{u_1(t)},e^{u_3(t)})+\theta(t)\eta(t)e^{u_2(t)}-b(t)e^{u_3(t)}+r(t)
\end{array}\right].
\end{equation}

In the following lemma we show that the linear map in~\eqref{eq:linear-map} is a Fredholm mapping of index zero
\begin{lemma}\label{lemma:Fredholm}
The linear map $\mathcal{L}$ in~\eqref{eq:linear-map} is a Fredholm mapping of index zero.
\end{lemma}

\begin{proof}
We have
\[
\begin{split}
\ker \mathcal{L}
& =\left\{ (u_1,u_2,u_3) \in X \cap C^1(\mathbb{R},\mathbb{R}^3): \frac{du_i(t)}{dt}=0, \ \ i=1,2,3 \right\}\\
& = \left\{ (u_1,u_2,u_3) \in X \cap C^1(\mathbb{R},\mathbb{R}^3): u_i \ \ \text{is constant}, \ \ i=1,2,3 \right\}
\end{split}
\]
and thus $\ker \mathcal{L}$ can be identified with $\R^3$. Therefore $\dim \ker \mathcal{L}=3$. On the other hand
\[
\begin{split}
\imagem \mathcal{L}
& =\left\{ (z_1,z_2,z_3) \in Z: \ \exists \ u \in X \cap C^1(\mathbb{R},\mathbb{R}^3): \frac{du_i(t)}{dt}=z_i(t), \ i=1,2,3 \right\}\\
& =\left\{ (z_1,z_2,z_3) \in Z: \int_0^\omega z_i(s)\ds=0, \ i=1,2,3\right\}.
\end{split}
\]
and any $z \in Z$ can be written as $z=\tilde{z}+\alpha$, where $\alpha=(\alpha_1,\alpha_2,\alpha_3) \in \R^3$ and $\tilde{z} \in \imagem \mathcal L$. Thus the complementary space of $\imagem \mathcal{L}$ consists of the constant functions. Thus, the complementary space has dimension 3 and therefore $\codim \imagem \mathcal{L}=3$.

Given any sequence $(z_n)$ in $\imagem \mathcal L$ such that
$$z_n=((z_1)_n,(z_2)_n,(z_3)_n) \to z=(z_1,z_2,z_3),$$
we have, for $i=1,2,3$ (note that $z \in Z$ since $Z$ is a Banach space and thus it is integrable in $[0,\omega]$ since it is continuous in that interval),
\[
\int_0^\omega z_i(s)\ds=\int_0^\omega \lim_{n \to +\infty} (z_i)_n(s)\ds=\lim_{n \to +\infty} \int_0^\omega (z_i)_n(s)\ds=0.
\]
Thus, $z \in \imagem \mathcal L$ and we conclude that $\imagem \mathcal L$ is closed in $Z$. Thus $\mathcal L$ is a Fredholm mapping of index zero.
\end{proof}

Consider the projectors $P:X \rightarrow X$ and $Q:Z \rightarrow Z$ given by
$$Pu(t)=\frac{1}{\omega}\int_{0}^{\omega}u(s)ds \quad \text{ and } \quad Qz(t)=\frac{1}{\omega}\int_{0}^{\omega}z(s)ds.$$
Note that  $\imagem P = \ker \mathcal{L}$ and that $\ker Q= \imagem (I-Q)= \imagem \mathcal{L}$.

Consider the generalized inverse of $\mathcal{L}$, $\mathcal K:\imagem \mathcal{L} \rightarrow D \cap \ker P$, given by $$\mathcal Kz(t)=\int_{0}^{t}z(s)ds-\frac{1}{\omega}\int_{0}^{\omega}\int_{0}^{r}z(s)\ds\dr$$
the operator $Q\mathcal{N}:X \rightarrow Z$ given by
\[
Q\mathcal{N}u(t)=\left[\begin{array}{l}
\frac{1}{\omega}\dint_{0}^{\omega}\Lambda(s)e^{-u_1(s)}-a(s)f(e^{u_1(s)},e^{u_3(s)})e^{u_3(s)}-\beta(s)e^{u_2(s)}\ds-\overline{\mu}\\[4mm]
\frac{1}{\omega}\dint_{0}^{\omega}\beta(s)e^{u_1(s)}-\eta(s)e^{u_3(s)}\ds-\overline{c}\\[4mm]
\frac{1}{\omega}\dint_{0}^{\omega}\gamma(s)a(s)f(e^{u_1(s)},e^{u_3(s)})e^{u_3(s)}+\theta(s)\eta(s)e^{u_2(s)}- b(s)e^{u_3(s)}\ds+\overline{r}
\end{array}\right]
\]
and the mapping  $\mathcal K(I-Q)\mathcal{N}:X \rightarrow D \cap \ker P$ given by
\[
\mathcal K(I-Q)\mathcal{N}u(t)=B_1(t)-B_2(t)-B_3(t),
\]
where
\[
\begin{split}
B_1(t)=\left[
\begin{array}{l}
\dint_{0}^{t}\Lambda(s)e^{-u_1(s)}-a(s)f(e^{u_1(s)},e^{u_3(s)})e^{u_3(s)}-\beta(s)e^{u_2(s)}-\mu(s)\ds\\[4mm]
\dint_{0}^{t}\beta(s)e^{u_1(s)}-\eta(s)e^{u_3(s)}-c(s)\ds\\[4mm]
\dint_{0}^{t}\gamma(s)a(s)f(e^{u_1(s)},e^{u_3(s)})e^{u_3(s)}+\theta(s)\eta(s)e^{u_2(s)}- b(s)e^{u_3(s)}dt+r(s)\ds
\end{array}\right],
\end{split}
\]
\small{
\[
B_2(t)=\left[\begin{array}{l}
\frac{1}{\omega}\dint_{0}^\omega \dint_{0}^r\Lambda(s)e^{-u_1(s)}-a(s)f(e^{u_1(s)},e^{u_3(s)})e^{u_3(s)}-\beta(s)e^{u_2(s)}-\mu(s)\ds\dr\\[4mm]
\frac{1}{\omega}\dint_{0}^\omega \dint_{0}^r\beta(s)e^{u_1(s)}-\eta(s)e^{u_3(s)}-c(s)\ds\dr\\[4mm]
\frac{1}{\omega}\dint_{0}^\omega \dint_{0}^r\gamma(s)a(s)f(e^{u_1(s)},e^{u_3(s)})e^{u_3(s)}+\theta(s)\eta(s)e^{u_2(s)}- b(s)e^{u_3(s)}+r(s)\ds\dr
\end{array}\right],
\]
}
and
\[
B_3(t)=\left(\frac{t}{\omega}-\frac{1}{2}\right)
\left[\begin{array}{l}
\dint_{0}^\omega\Lambda(s)e^{-u_1(s)}-a(s)f(e^{u_1(s)},e^{u_3(s)})e^{u_3(s)}-\beta(s)e^{u_2(s)}-\mu(s)\ds\\[4mm]
\dint_{0}^\omega\beta(s)e^{u_1(s)}-\eta(s)e^{u_3(s)}-c(s)\ds\\[4mm]
\dint_{0}^\omega\gamma(s)a(s)f(e^{u_1(s)},e^{u_3(s)})e^{u_3(s)}+\theta(s)\eta(s)e^{u_2(s)}- b(s)e^{u_3(s)}+r(s)\ds
\end{array}\right].
\]

The next lemma shows that $\mathcal{N}$ is $\mathcal L$-compact in the closure of any open bounded subset of its domain.
\begin{lemma}
The map $\mathcal{N}$ is $\mathcal L$-compact in the closure of any open bounded set $U \subseteq X$.
\end{lemma}
\begin{proof}
Let $U \subseteq X$ be an open bounded set and $\overline{U}$ its closure in $X$. Then, there is $M>0$ such that, for any $u=(u_1,u_2,u_3) \in \overline{U}$, we have that $|u_i(t)| \leqslant M$, $i=1,2,3$. Letting $Q\mathcal{N}u=((Q\mathcal{N})_1u,(Q\mathcal{N})_2u,(Q\mathcal{N})_3u)$, we have
\[
\left|(Q\mathcal{N})_1u(t)\right|\le \e^M \left(\bar\Lambda+\bar a f(e^M,0)+\bar\beta\right)+\bar\mu,
\]
\[
\left|(Q\mathcal{N})_2u(t)\right|\le e^M(\bar\beta+\bar\eta)+\overline{c}
\]
\[
\left|(Q\mathcal{N})_3u(t)\right|\le \e^M\left(\overline{\gamma a}f(e^M,0)+\overline{\theta\eta}+ \bar b\right)+\overline{r}
\]
and we conclude that $Q\mathcal{N}(\overline{U})$ is bounded.

Let now $$\mathcal K(I-Q)\mathcal{N}u =\left((\mathcal K(I-Q)\mathcal{N})_1u,(\mathcal K(I-Q)\mathcal{N})_2u,(\mathcal K(I-Q)\mathcal{N})_3u\right).$$
Let $B \subset X$ be a bounded set. Note that the boundedness of $B$ implies that there is $M$ such that $|u_i|<M$, for all $i=1,2,3$, and all $u=(u_1,u_2,u_3) \in B$. It is immediate that $\{\mathcal K(I-Q)\mathcal{N}u:u \in B\}$ is pointwise bounded. Given $u=(u_1,u_2,u_3)_{n \in\N} \in B$ we have
\begin{equation}\label{eq:equicont-1}
\begin{split}
& (\mathcal K(I-Q)\mathcal{N})_1u(t)-(\mathcal K(I-Q)\mathcal{N})_1u(v))\\
= &\dint_v^t \Lambda(s)e^{-u_1(s)}-a(s)f(e^{u_1(s)},e^{u_2(s)})e^{u_2(s)}-\beta(s)e^{u_2(s)}-\mu(s)\ds\\
&-\frac{t-v}{\omega}\dint_{0}^\omega\Lambda(s)e^{-u_1(s)}-a(s)f(e^{u_1(s)},e^{u_2(s)})e^{u_2(s)}-\beta(s)e^{u_2(s)}-\mu(s)\ds\\
\le & 2(t-v)\left[e^M(\Lambda^u+a^u f(e^M,0)+\beta^u e^M)+\mu^M\right],
\end{split}
\end{equation}
and similarly
\begin{equation}\label{eq:equicont-2}
\begin{split}
& (\mathcal K(I-Q)\mathcal{N})_2u(t)-(\mathcal K(I-Q)\mathcal{N})_2u(v)
\le2(t-v)\left[e^M(\beta^u+\eta^u)+c^u\right]
\end{split}
\end{equation}
and
\begin{equation}\label{eq:equicont-3}
\begin{split}
& (\mathcal K(I-Q)\mathcal{N})_3u(t)-(\mathcal K(I-Q)\mathcal{N})_3u(v))\\
& \le
2(t-v)\left[(\gamma^ua^u f(e^M,0)+\theta^u\eta^u+ b^u)e^M+r^u\right].
\end{split}
\end{equation}
By~\eqref{eq:equicont-1},~\eqref{eq:equicont-2} and~\eqref{eq:equicont-3}, we conclude that $\{\mathcal K(I-Q)\mathcal{N}u:u \in B\}$ is equicontinuous.
Therefore, by Ascoli-Arzela's theorem, $\mathcal K(I-Q)\mathcal{N}(B)$ is relatively compact. Thus the operator $\mathcal K(I-Q)\mathcal{N}$ is compact.

We conclude that $\mathcal{N}$ is $\mathcal L$-compact in the closure of any bounded set contained in $X$.
\end{proof}

\subsection{Application of Mawhin's continuation theorem.}\label{subsection:Proof-subsection3}
In this section we will construct the set where, applying Mahwin's continuation theorem, we will find the periodic orbit in the statement of our result.

Consider the system of algebraic equations:
\begin{equation}\label{eq:principal-aux-3}
\begin{cases}
\overline{\Lambda}e^{-u_1}-\overline{a}f(e^{u_1},e^{u_3})e^{u_3-u_1}-\overline{\beta}e^{u_2}-\overline{\mu}=0\\
\overline{\beta}e^{u_1}-\overline{\eta}e^{u_3}-\overline{c}=0\\
\overline{\gamma a}f(e^{u_1},e^{u_3})+ \overline{\theta\eta}e^{u_2}-\overline{b}e^{u_3}+\overline{r}=0
\end{cases}.
\end{equation}
By the second and third equations we get
\[
\e^{u_1}=\frac{\overline{\eta}\e^{u_3}+\overline{c}}{\overline{\beta}}
\quad \text{ and } \quad
\e^{u_2}=-\frac{\overline{\gamma a}}{\overline{\theta\eta}}f\left(\frac{\overline{\eta}\e^{u_3}+\overline{c}}{\overline{\beta}},e^{u_3}\right)
+\frac{\overline{b}}{\overline{\theta\eta}}e^{u_3}-\frac{\overline{r}}{\overline{\theta\eta}}.
\]
Therefore, using the first equation, we get
$$G_1(\e^{u_3})-G_2(\e^{u_3})f\left(\frac{\overline{\eta}\e^{u_3}+\overline{c}}{\overline{\beta}},e^{u_3}\right)
-G_3(\e^{u_3})=0,$$
where
$$
G_1(x)=\dfrac{\overline{\Lambda}\,\overline{\beta}}{\overline{\eta}x+\overline{c}},
\quad G_2(x)=\dfrac{\overline{a}\overline{\beta}x}{\overline{\eta}x+\overline{c}}-\dfrac{\overline{\beta}\overline{\gamma a}}{\overline{\theta \eta}} \quad \text{ and } \quad G_3(x)=\dfrac{\overline{\beta}\, \overline{b}}{\overline{\theta \eta}}x-\dfrac{\overline{\beta}\overline{r}}{\overline{\theta \eta}}+\overline{\mu}
$$
Consider the function $G:[\overline{r}/\overline{b},+\infty[\to \R$ (notice that, by the third equation in~\eqref{eq:principal-aux-3},
we have $e^{u_3}\ge\overline{r}/\overline{b}$), given by
$$G(x)=G_1(x)-G_2(x)f\left(\frac{\overline{\eta}x+\overline{c}}{\overline{\beta}},x\right)
-G_3(x)$$
and observe that function $G_1$ is decreasing and functions $G_2$ and $G_3$ are increasing. Thus, since the function $[\overline{r}/\overline{b},+\infty[ \ni x \mapsto f((\overline{\eta}x+\overline{c})/\overline{\beta},x)$ is nondecreasing, we conclude that
$$[\overline{r}/\overline{b},+\infty[ \ni x \mapsto -G_2(x)f((\overline{\eta}x+\overline{c})/\overline{\beta},x)$$
is a decreasing function. It is immediate the function $[\overline{r}/\overline{b},+\infty[ \ni x \mapsto -G_3(x)$ is decreasing. Consequently $G$ is a decreasing function and equation~\eqref{eq:principal-aux-3} has, at most, one solution.

It is easy to verify that
$$\lim_{x \to +\infty} G(x)=-\infty$$
and, by the hypothesis in our theorem
\[
\begin{split}
G(\overline{r}/\overline{b})
& =\dfrac{\overline{\Lambda}\,\overline{\beta}}{\overline{\eta}\,\overline{r}/\overline{b}+\overline{c}}
-\left(\dfrac{\overline{a}\overline{\beta}\overline{r}/\overline{b}}{\overline{\eta}\,\overline{r}/\overline{b}+\overline{c}}-\dfrac{\overline{\beta}\overline{\gamma a}}{\overline{\theta \eta}}\right)f\left(\frac{\overline{\eta}\,\overline{r}/\overline{b}+\overline{c}}{\overline{\beta}},\overline{r}/\overline{b}\right)
-\overline{\mu}\\
& = \overline{\mu}\left(\overline{\mathcal R}_0-1-\dfrac{\overline{\beta}}{\overline{\mu}}\left(\dfrac{\overline{a}\,\overline{r}}{\overline{\eta}\,\overline{r}+\overline{c}\overline{b}} -\dfrac{\overline{\gamma a}}{\overline{\theta \eta}}\right) f\left(\frac{\overline{\eta}\,\overline{r}/\overline{b}+\overline{c}}{\overline{\beta}},\overline{r}/\overline{b}\right)
\right)>0.
\end{split}
\]
Thus we conclude that there is a unique solution of equation~\eqref{eq:principal-aux-3}. Denote this solution by $p^*(t)=(p_1^*,p_2^*,p_3^*)$.

By Lemmas~\ref{lemma:subsection-persist-1}, \ref{lemma:subsection-persist-2} and \ref{lemma:subsection-persist-3}, there is a constant $M_0>0$ such that $\|u_\lambda(t)\|<M_0$, for any $t \in [0,\omega]$ and any periodic solution $u_\lambda(t)$ of~\eqref{eq:principal-aux}. Let
\begin{equation}\label{eq:set-Mawhin}
U=\{(u_1,u_2,u_3) \in X: \|(u_1,u_2,u_3)\|<M_0+\|p^*\| \}.
\end{equation}
Conditions M1. and M2. in Mawhin's continuation theorem (see appendix~\ref{appendix:MCT}) are fulfilled in the set $U$ defined in~\eqref{eq:set-Mawhin}.

Using the notation $v=(\e^{p^*_1},\e^{p^*_3})$, the Jacobian matrix of the vector field corresponding to~\eqref{eq:principal-aux-3} computed in $(p_1^*,p_2^*,p_3^*)$ is
\[
J=
\left[
\begin{array}{ccc}
-\overline{a}\frac{\partial f}{\partial S}(v)\e^{p_3^*}-\overline{\beta}\e^{p_2^*}-\overline{\mu} & -\overline{\beta}\e^{p_2^*} &
-\overline{a}\frac{\partial f}{\partial P}(v)\e^{2p_3^*-p_1^*}-\overline{a}f(v)\e^{p_3^*-p_1^*}\\
\overline{\beta}\e^{p_1^*} & 0 & -\overline{\eta}\e^{p_3^*}\\
\overline{\gamma a}\frac{\partial f}{\partial S}(v)\e^{p_1^*} &
\overline{\theta \eta}\e^{p_2^*} & \overline{\gamma a}\frac{\partial f}{\partial P}(v)\e^{p_3^*}-\overline{b}\e^{p_3^*}
\end{array}
\right].
\]
Since we are assuming that $\overline{\beta}\overline{\gamma a}-\overline{\theta\eta}\overline{a}\le 0$, we have
\[
\begin{split}
& \det J(p_1^*,p_2^*,p_3^*)\\
=&  -\overline{\beta}\e^{p_1^*} \left(-\overline{\beta}\e^{p_2^*}\left(\overline{\gamma a}\frac{\partial f}{\partial P}(v)\e^{p_3^*}-\overline{b}\e^{p_3^*}\right)+\left(\overline{a}\frac{\partial f}{\partial P}(v)\e^{2p_3^*-p_1^*}+\overline{a}f(v)\e^{p_3^*-p_1^*}\right) \overline{\theta \eta}\e^{p_2^*} \right)\\
&+\overline{\eta}\e^{p_3^*}\left(
\left(-\overline{a}\frac{\partial f}{\partial S}(v)\e^{p_3^*}-\overline{\beta}\e^{p_2^*}-\overline{\mu}\right)\overline{\theta \eta}\e^{p_2^*}
+\overline{\beta}\e^{p_2^*}\overline{\gamma a}\frac{\partial f}{\partial S}(v)\e^{p_1^*}\right)\\
=&  -\overline{\beta}\e^{p_2^*+p_3^*} \left(-\overline{\beta}\overline{\gamma a}\frac{\partial f}{\partial P}(v)\e^{p_1^*}+\overline{\beta}\overline{b}\e^{p_1^*}
+\overline{\theta \eta}\overline{a}\frac{\partial f}{\partial P}(v)\e^{p_3^*}+\overline{\theta \eta}\overline{a}f(v)\right) \\
&+\overline{\eta}\e^{p_3^*+p_2^*}\left(
-\overline{a}\frac{\partial f}{\partial S}(v)\e^{p_3^*}\overline{\theta \eta}-\overline{\beta}\e^{p_2^*}\overline{\theta \eta}-\overline{\mu}\overline{\theta \eta}
+\overline{\beta}\overline{\gamma a}\frac{\partial f}{\partial S}(v)\e^{p_1^*}\right)\\
=&  (\overline{\beta}\overline{\gamma a}-\overline{\theta\eta}\overline{a})\left(\overline{\beta}\frac{\partial f}{\partial P}(v)+\overline{\eta}\frac{\partial f}{\partial S}(v)\right) -\overline{\beta}\e^{p_2^*+p_3^*} \left(\overline{\beta}\overline{b}\e^{p_1^*}
+\overline{\theta \eta}\overline{a}f(v)\right) \\
&-\overline{\eta}\e^{p_3^*+p_2^*}\left(
\overline{\beta}\e^{p_2^*}\overline{\theta \eta}+\overline{\mu}\overline{\theta \eta}\right)<0.\\
\end{split}
\]
Let $\mathcal{I}: \text{Im} Q \rightarrow \text{ker} \mathcal{L}$ be an isomorphism. Thus
\begin{equation}\label{eq:cond-3-Mawhin}
\deg(\mathcal I\mathcal Q\mathcal N, U\cap\ker\mathcal L,0) = \det J(p_1^*,p_2^*,p_3^*) \ne 0
\end{equation}
and condition M3) in Mawhin's continuation theorem (see appendix~\ref{appendix:MCT}) holds.
Taking into account Lemma~\ref{lemma:subsection-persist-3}, the proof of Theorem~\ref{the:main} is completed.

\section{Examples.}\label{section:examples}
In this section we present some examples to illustrate our main result.
\subsection{A model with no predation on susceptible preys.}\label{subsection:example-no-predation-susceptible-prey}
Letting $a(t)\equiv 0$ or $f \equiv 0$ in system~\eqref{eq:principal}, and still assuming that the real valued functions $\Lambda$, $\mu$, $\beta$, $\eta$, $c$, $\gamma$, $r$, $\theta$ and $b$ are periodic with period $\omega$, nonnegative, continuous and also that $\bar\Lambda>0$, $\bar\mu>0$, $\bar r>0$ and $\bar b>0$, we obtain the periodic model considered in~\cite{Silva-JMAA-2017}:
\begin{equation}\label{eq:principal-example-no-predation-susceptible-prey}
\begin{cases}
S'=\Lambda(t)-\mu(t)S-\beta(t)SI\\
I'=\beta(t)SI-\eta(t)PI-c(t)I\\
P'=(r(t)-b(t)P)P+\theta(t)\eta(t)PI
\end{cases}
\end{equation}
Note that conditions S\ref{cond-1}) and S\ref{cond-7}) are assumed and conditions S\ref{cond-2}) to S\ref{cond-6}) and S\ref{cond-8}) are trivially satisfied since $f\equiv 0$.
Also note that system~\eqref{eq:auxiliary-system-SP} becomes in this context
\begin{equation}\label{eq:auxiliary-system-example-no-predation-susceptible-prey}
\begin{cases}
x'=\lbd(\Lambda(t)-\mu(t)x-\eps_1 x)\\
z'=\lbd(r(t)- b(t)z+\eps_2)z
\end{cases}
\end{equation}
and, by Lemmas 1 to 4 in~\cite{Niu-Zhang-Teng-AMM-2011} we conclude that condition~S\ref{cond-9}) holds in this setting.
Note also that condition~\eqref{cond:teo} becomes $\overline{\cR}_0>1$ and condition $\overline{\gamma a} \,\overline{\beta}-\overline{\theta\eta}\overline{a}\le 0$ is trivially satisfied since we can take $\gamma=0$. We obtain the following corollary that recovers the result in~\cite{Silva-JMAA-2017}:
\begin{corollary}
If $\widetilde{\cR}_0>1$ and $\overline{R}_0>1$ hold, then system~\eqref{eq:principal-example-no-predation-susceptible-prey} possesses an endemic periodic orbit of period $\omega$.
\end{corollary}

\subsection{A model with Holling-type I functional response.}\label{subsection:example-Holling-type I}
Letting $f(S,P) = S$ (Holling-type I functional response) in system~\eqref{eq:principal}, and assuming that the real valued functions $\beta$, $\eta$ e $c$ are periodic with period $\omega$ and that the real valued functions $\Lambda$, $\mu$, $\gamma$, $r$, $\theta$, $b$ and $a$ are constant and positive, we obtain the periodic model:
\begin{equation}\label{eq:principal-Holling-type I}
\begin{cases}
S'=\Lambda-\mu S-aSP+\beta(t) SI\\
I'=\beta(t) SI-\eta(t) PI-c(t)I\\
P'=(r - bP)P+\gamma aSP+\theta \eta(t) PI
\end{cases}
\end{equation}
Since $f(S,P)=S$, conditions S\ref{cond-2}) to S\ref{cond-6}) are trivially satisfied. Conditions S\ref{cond-1}) and S\ref{cond-7}) are assumed  and S\ref{cond-8}) is satisfied with $K=\alpha=1$.
Notice additionally that system~\eqref{eq:auxiliary-system-SP} becomes in our context
\begin{equation}\label{eq:auxiliary-system-Holling-type I}
\begin{cases}
x'=\lbd(\Lambda-\mu x-a xz-\eps_1 x)\\
z'=\lbd(r-bz +\gamma a x +\eps_2)z
\end{cases}.
\end{equation}
System~\eqref{eq:auxiliary-system-Holling-type I} has two equilibriums: $E_1=(\Lambda/(\mu+\eps_1),0)$ and
$$E_2=\left(\dfrac{\sqrt{V^2+4\Lambda\gamma a^2/b}-V}{2\gamma a^2/b}, \, \dfrac{\sqrt{V^2+4\Lambda\gamma a^2/b}-V}{2\gamma a^2/b}+r+\eps_2\right),$$
where $V=\mu+\eps_1+a(r+\eps_2)/b$. It is easy to check that $E_2$ is locally attractive and that $E_1$ is a saddle point whose stable manifold coincides with the x-axis. If $0<\alpha<(r+\eps_2)/b$ then, in the line $z=\alpha$ the flow points upward. Additionally, if $\Lambda<\mu(\mu+\eps_1)/a$, in the line $x=\mu/a$ the flow points to the left and the $x$-coordinate of $E_1$ is less than $\mu/a$. Thus the region $R=\{(x,z) \in \R^2: 0\le x \le \mu/a \ \wedge \ z \ge \alpha \}$ is positively invariant. Since the divergence of the vector field is given by $-\mu-\eps_1+\eps_2-(a+2b)z+\gamma ax$, we conclude that it is null on the line $z=\frac{-\mu-\eps_1+\eps_2}{a+2b}+\frac{\gamma a}{a+2b}x$. Thus the divergence of the vector field doesn't change sign on the region $R$ and this forbids the existence of a peridic orbit on $R$. There is also no periodic orbit on $(\R_0^+)^2 \setminus R$ since there is no additional equilibrium in $(\R_0^+)^2$. Since $E_2$ is locally asymptotically stable, there is no homoclinic orbit conecting $E_2$ to itself. Therefore, the $\omega$-limit of any orbit in $(\R^2)^+$ must be the equilibrium point $E_2$ and the global asymptotic stability of~\eqref{eq:auxiliary-system-Holling-type I} for sufficiently small $\eps_1,\eps_2>0$ follows. We conclude that condition~S\ref{cond-9}) holds.

Notice that condition~\eqref{cond:teo} becomes $\overline{\cR}_0>1+\dfrac{a(r \theta\bar{\eta}-\bar{\gamma}\bar{\eta} r - \bar{\gamma}\bar{c}b)}{\mu\bar{\eta}\theta b}$ and condition ${\gamma}\overline{\beta}-\theta\overline{\eta} \le 0$ is trivially satisfied. We obtain the corollary that generalizes the result in~\cite{Silva-JMAA-2017}:
\begin{corollary}\label{cor}
If $\Lambda<\mu^2/a$, ${\gamma}\overline{\beta}-\theta\overline{\eta} \le 0$, $\widetilde{\cR}_0>1$ and
$$\overline{R}_0>1+\dfrac{a(r \theta\bar{\eta}-\bar{\gamma}\bar{\eta} r - \bar{\gamma}\bar{c}b)}{\mu\bar{\eta}\theta b}$$
then system~\eqref{eq:principal-Holling-type I} possesses an endemic periodic orbit of period $\omega$.
\end{corollary}

To do some simulation, we consider the following particular set of parameters: $\Lambda=0.1$; $\mu=0.6$; $\beta(t)=20(1+0.9\cos(2\pi t))$; $\eta(t)=0.7(1+0.7\cos(\pi+2\pi t))$; $c(t)=0.1$; $r=0.2$; $b=0.3$; $\theta=10$, $\gamma(t)=0.1$ and $a=3$. We obtain the model:
\begin{equation}\label{eq:Periodic-Holling-type I}
\begin{cases}
S'=0.1-0.6S-20(1+0.9\cos(2\pi t))SI-3SP\\
I'=20(1+0.9\cos(2\pi t))SI-0.7(1+0.7\cos(\pi+2\pi t))PI-0.1I\\
P'=(0.2-0.3P)P+7(1+0.7\cos(\pi+2\pi t))PI+0.3SP
\end{cases},
\end{equation}

Notice that, for our model, $\Lambda=0.1>0.012=\mu^2/a$, ${\gamma}\overline{\beta}-\theta\overline{\eta} =-5<0$, $\overline{R}_0\approx 5.88 >1 + 3.3$ and $\tilde{R}_0\approx 24.8 >1$, and thus the conditions in Corollary~\ref{cor} are fulfilled. Considering the initial condition $(S_0,I_0,P_0)=(0.03567,0.02047,0.88021)$ we obtain the periodic orbit in figure~\ref{fig_periodic2-Holling-type I}.
\begin{figure}[h]
  \begin{minipage}[b]{.32\linewidth}
    \includegraphics[width=\linewidth]{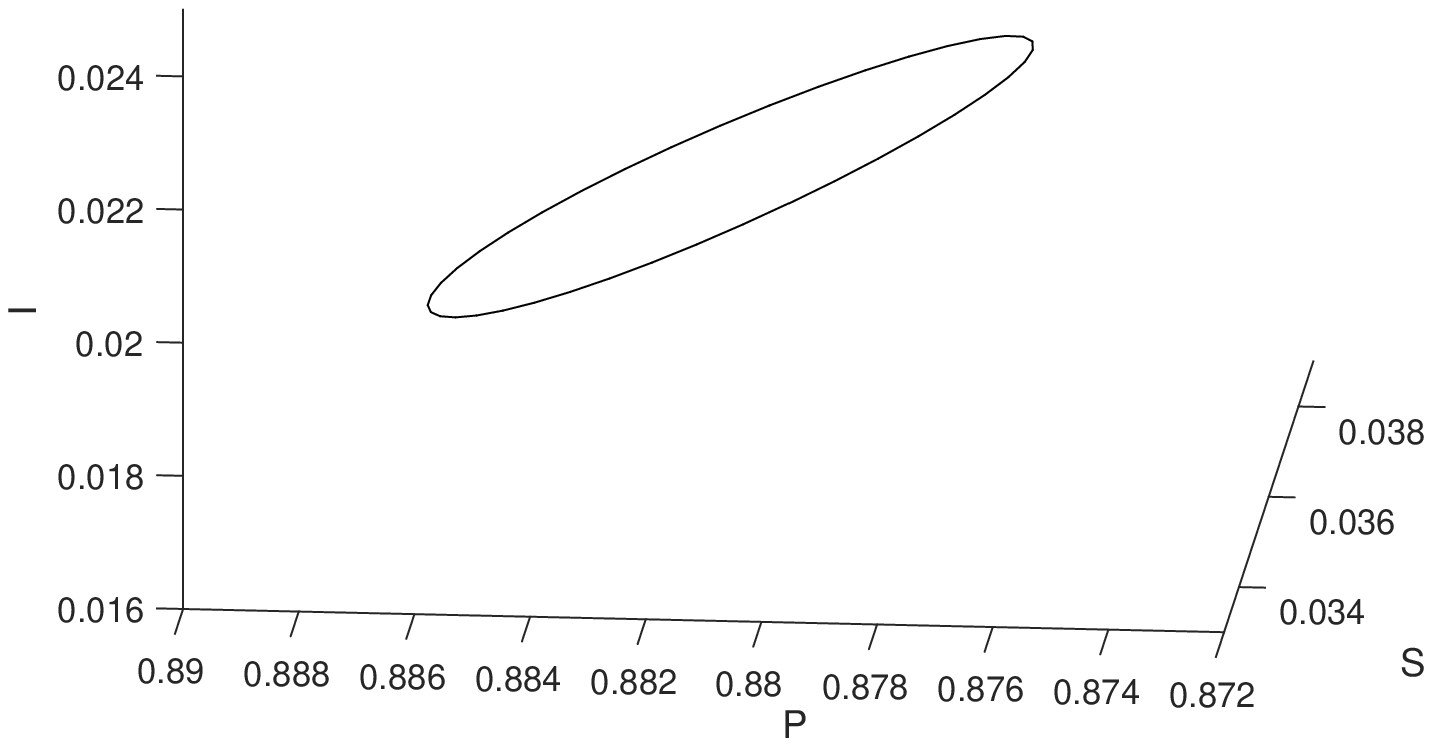}
  \end{minipage}
  \begin{minipage}[b]{.32\linewidth}
    \includegraphics[width=\linewidth]{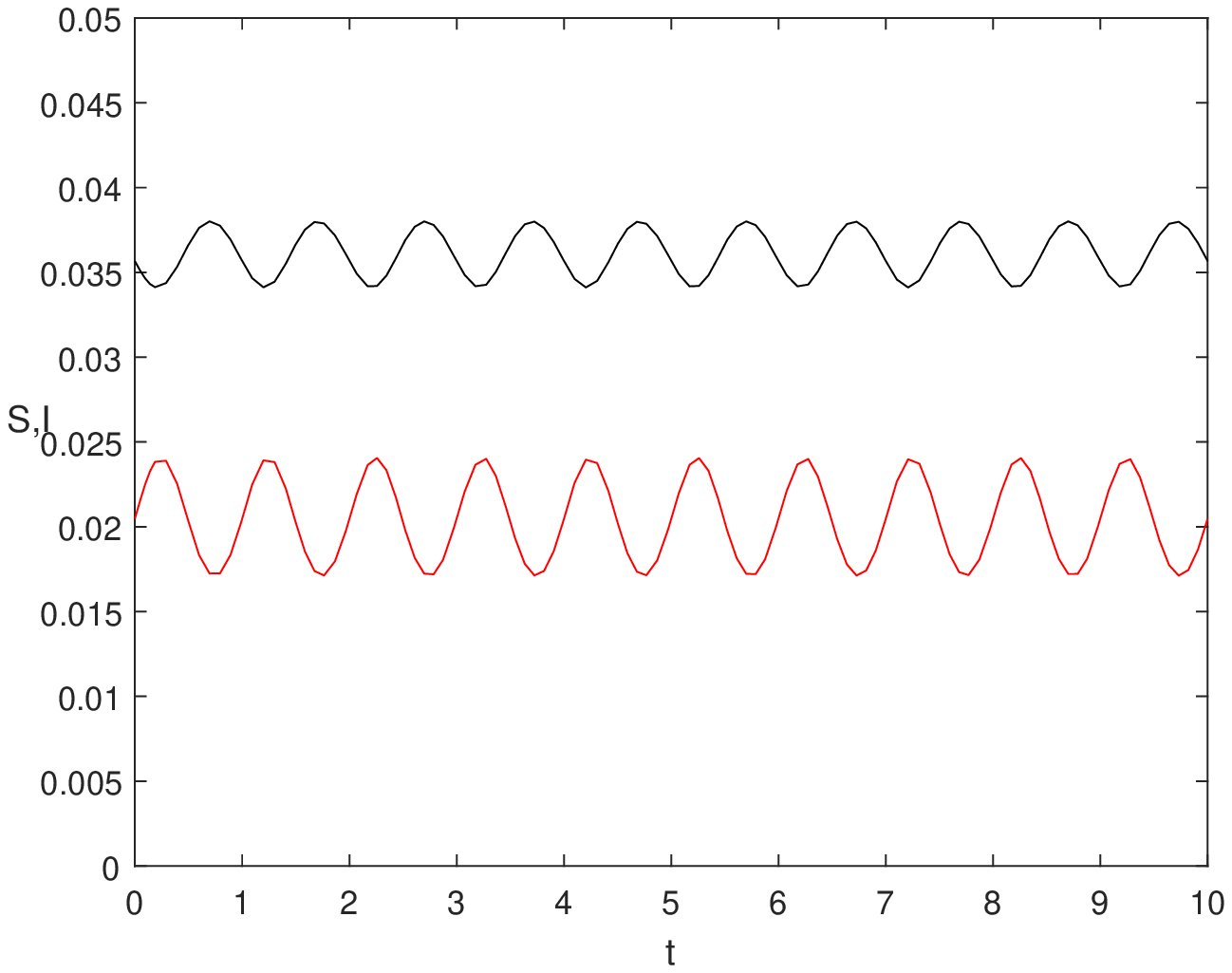}
  \end{minipage}
  \begin{minipage}[b]{.32\linewidth}
    \includegraphics[width=\linewidth]{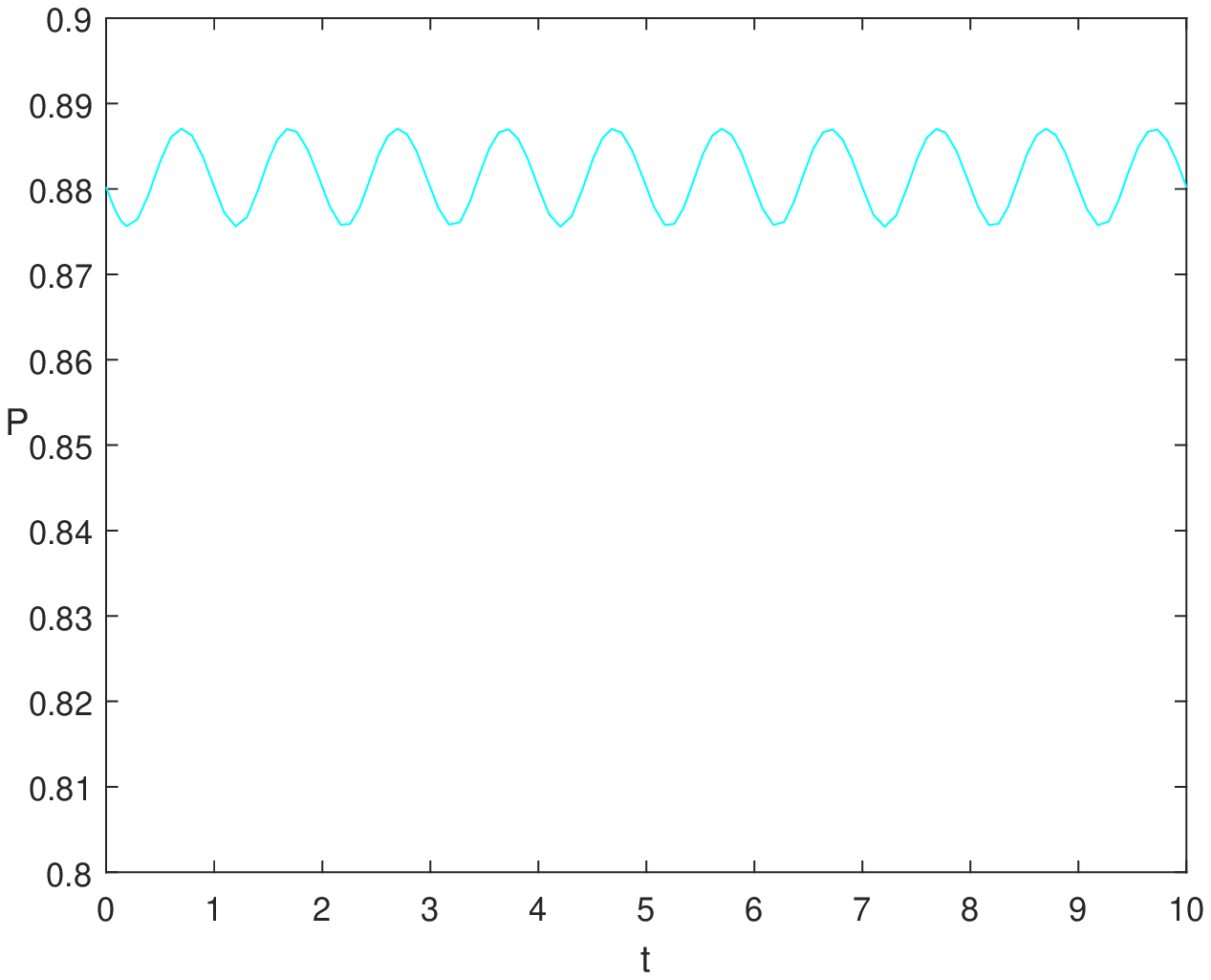}
  \end{minipage}
    \caption{Periodic orbit for model~\eqref{eq:Periodic-Holling-type I}}
      \label{fig_periodic2-Holling-type I}
\end{figure}
Although our theoretical result doesn't imply the attractivity of the periodic solution, the simulations carried out suggest that this is the case.


\appendix
\section{Mawhin's continuation theorem}\label{appendix:MCT}

In this appendix we state Mawhin's continuation theorem~\cite[Part IV]{Mawhin}. Let $X$ and $Z$ be Banach spaces.
\begin{definition}
A linear map $\mathcal{L}: D \subseteq X \rightarrow Z$ is called a Fredholm mapping of index zero if
\begin{enumerate}
\item $\dim \ker \mathcal{L}= \codim \imagem \mathcal{L} \leqslant \infty$;
\item $\imagem \mathcal{L}$ is closed in $Z$.
\end{enumerate}
\end{definition}
Given a Fredholm mapping of index zero $\mathcal{L}: D \subseteq X \rightarrow Z$ it is well known that there are continuous projectors $P:X \rightarrow X$ and $Q:Z \rightarrow Z$ such that:
\begin{enumerate}
\item $\imagem P=\ker \mathcal{L}$;
\item $\ker Q=\imagem \mathcal{L}=\imagem (I-Q)$;
\item $X=\ker \mathcal{L} \oplus \ker P$;
\item $Z=\imagem \mathcal{L} \oplus \imagem Q$.
\end{enumerate}
It follows that $\mathcal{L}|_{D\bigcap \ker P}:(I-P)X \rightarrow \imagem \mathcal{L}$ is invertible. We denote the inverse of that map by $\mathcal K$.

\begin{definition}
A continuous mapping $\mathcal{N}: X \rightarrow Z$ is called $\mathcal L$-compact on $\overline{U} \subset X$, where $U$ is an open bounded set, if
\begin{enumerate}
\item $QN(\overline{U})$ is bounded;
\item $\mathcal K(I-Q)\mathcal{N}:\overline{U} \rightarrow X$ is compact.
\end{enumerate}
\end{definition}
Note that, since Im $Q$ is isomorphic to ker $\mathcal{L}$, there is an isomophism $\mathcal{I}: \text{Im} Q \rightarrow \text{ker} \mathcal{L}$.
We are now prepared to state the Mawhin's continuation theorem.

\begin{theorem}[Mawhin's continuation theorem]
Let $X$ and $Z$ be Banach spaces and let $U \subset X$ be an open set. Assume that $\mathcal{L}: D \subseteq X \rightarrow Z$
 is a Fredholm mapping of index zero and let $\mathcal{N}: X \rightarrow Z$ be $\mathcal L$-compact on $\overline{U}$. Additionally, assume that
\begin{enumerate}
\item[M1)] for each $\lambda \in (0,1)$ and $x \in \partial U \cap D$ we have $\mathcal{L}x \neq \lambda\mathcal{N}x$;
\item[M2)] for each $x \in \partial U \cap \ker \mathcal{L}$ we have $Q\mathcal{N}x \neq 0$;
\item[M3)] $\deg(\mathcal{I}Q\mathcal{N},U \cap \ker \mathcal{L},0) \neq 0$.
\end{enumerate}
Then the operator equation $\mathcal{L}x=\mathcal{N}x$ has at least one solution in $D \cap \overline{U}$.
\end{theorem}

\bibliographystyle{elsart-num-sort}

\end{document}